\documentclass[12pt]{extarticle}

\usepackage[margin=1.5in]{geometry}  
\usepackage{graphicx}              
\usepackage{amsmath}               
\usepackage{amsfonts}              
\usepackage{amsthm}                
\usepackage{framed}

\newtheorem{thm}{Theorem}
\newtheorem{lem}[thm]{Lemma}

\begin{document}

\nocite{*}

\title{\bf On the Riemann Hypothesis and the Difference Between Primes}

\author{\textsc{Adrian W. Dudek} \\ 
Mathematical Sciences Institute \\
The Australian National University \\ 
\texttt{adrian.dudek@anu.edu.au}}
\date{}

\maketitle

\begin{abstract}
\noindent We prove some results concerning the distribution of primes on the Riemann hypothesis. First, we prove the explicit result that there exists a prime in the interval $(x-\frac{4}{\pi} \sqrt{x} \log x,x]$ for all $x \geq 2$; this improves a result of Ramar\'{e} and Saouter. We then show that the constant $4/\pi$ may be reduced to $(1+\epsilon)$ provided that $x$ is taken to be sufficiently large. From this we get an immediate estimate for a well-known theorem of Cram\'{e}r, in that we show the number of primes in the interval

$$(x, x+c \sqrt{x} \log x]$$
is greater than $\sqrt{x}$ for $c=3+\epsilon$ and all sufficiently large $x$.
\end{abstract}

\section{Introduction}

Much is already known on the interplay between the zeroes of the Riemann zeta-function $\zeta(s)$ and the distribution of prime numbers; one can see Ingham's well-known text \cite{inghambook} for more details. The Riemann hypothesis, which asserts that all of the non-trivial zeroes of $\zeta(s)$ have real part of $1/2$, thus presents itself as an important problem in number theory.

On the assumption of the Riemann hypothesis, von Koch \cite{vonkoch} proved that there exists a constant $k$ such that the interval $(x-k \sqrt{x} \log^2 x, x)$ contains a prime for all $x \geq x_0$. Schoenfeld \cite{schoenfeldjust} made this result precise, showing that one can take $K = 1/(4 \pi)$ and $x_0 = 599$. 

Cram\'{e}r \cite{cramer} improved the result of von Koch by proving the following theorem.

\begin{thm} \label{cramer}
Suppose the Riemann hypothesis is true. Then it is possible to find a positive constant $c$ such that

\begin{equation} \label{cramer1}
\pi(x+c \sqrt{x} \log x) - \pi(x) > \sqrt{x}
\end{equation}
for $x \geq 2$. Thus if $p_n$ denotes the $n$th prime, we have

\begin{equation} \label{cramer2}
p_{n+1} - p_n = O(\sqrt{p_n} \log p_n).
\end{equation}
\end{thm}
Goldston \cite{goldston} made this result more precise by showing that one could take $c=5$ in the above theorem for all sufficiently large values of $x$. He also showed that 

$$p_{n+1} - p_n < 4 p_n^{1/2} \log p_n$$
for all sufficiently large values of $n$. It should be noted that Goldston was not trying in any way to find the optimal constants; he was providing a new proof of Cram\'{e}r's theorem. Ramar\'{e} and Saouter \cite{ramaresaouter} made this result explicit by showing that for all $x \geq 2$ there exists a prime in the interval $(x-\frac{8}{5} \sqrt{x} \log x, x]$.

The first purpose of this paper is to give the following improvement on the work of Ramar\'{e} and Saouter.

\begin{thm} \label{one}
Suppose the Riemann hypothesis is true. Then there is a prime in the interval $(x-\frac{4}{\pi} \sqrt{x} \log x,x]$ for all $x\geq 2$.

\end{thm}

We prove this theorem using a weighted version of the Riemann von-Mangoldt explicit formula and some standard explicit estimates for sums over the zeroes of the Riemann zeta-function. It should be noted that the constant $4/\pi$ appearing in the above theorem is not optimal. The question of the optimal constant in Theorem \ref{one} is thus an open problem. To this end, we prove the following theorem.

\begin{thm} \label{implicit}
Suppose the Riemann hypothesis is true and let $\epsilon > 0$. Then there is a prime in the interval $(x-(1+\epsilon) \sqrt{x} \log x,x]$ for all sufficiently large values of $x$.

\end{thm}

It is not clear to the author whether the optimal constant is $1$ or something smaller. The reader may wish to see the work of Goldston and Heath-Brown \cite{goldstonheathbrown}, for they show that one has an arbitrarily small constant on some more sophisticated conjectures.

From our proof of Theorem \ref{implicit}, it follows readily that Theorem \ref{cramer} can be taken with $c = 3 + \epsilon$ for sufficiently large values of $x$. It is clear from the prime number theorem that $c>1$.

\section{Estimates on the Riemann hypothesis}

\subsection{A smooth explicit formula}

The purpose of this section is to prove Theorem \ref{one}. We define the von Mangoldt function as

\begin{displaymath}
   \Lambda(n) = \left\{
     \begin{array}{ll}
       \log p  & : \hspace{0.1in} n=p^m, \text{ $p$ is prime, $m \in \mathbb{N}$}\\
       0   & : \hspace{0.1in} \text{otherwise}
     \end{array}
   \right.
\end{displaymath} 
and introduce the sum $\psi(x) = \sum_{n \leq x} \Lambda(n)$. This summatory function submits itself to the Riemann von-Mangoldt explicit formula

\begin{equation} \label{explicitoriginal}
\psi(x) = x - \sum_{\rho} \frac{x^\rho}{\rho}-\log 2\pi - \frac{1}{2} \log(1-x^{-2})
\end{equation}
where $x>0$ is not an integer and the sum is over all nontrivial zeroes $\rho = \beta+i\gamma$ of the Riemann zeta-function. We define the weighted sum

$$\psi_1 (x) = \sum_{n \leq x} (x-n) \Lambda(n) = \int_2^x \psi(t) dt$$
and prove an analogous explicit formula.

\begin{lem} \label{first}
For $x>0$ and $x \notin \mathbb{Z}$ we have

\begin{equation} \label{explicit}
\psi_1(x) = \frac{x^2}{2} - \sum_{\rho} \frac{x^{\rho+1}}{\rho (\rho +1)} - x \log(2\pi) + \epsilon(x)
\end{equation}
where

$$|\epsilon(x)| < \frac{12}{5}.$$
\end{lem}

\begin{proof}
We integrate both sides of (\ref{explicitoriginal}) over the interval $(2,x)$ to get

$$\psi_1(x) = \frac{x^2}{2} - \sum_{\rho} \frac{x^{\rho+1}}{\rho (\rho +1)} - x \log(2\pi) + \epsilon(x)$$
where

$$|\epsilon(x)| < 2 + \bigg| \sum_{\rho} \frac{2^{\rho+1}}{\rho (\rho+1)} \bigg| + \frac{1}{2} \bigg| \int_2^x \log( 1- t^{-2} ) dt\bigg|.$$

The integral can be evaluated to yield $\log(16/27)$, and the sum over the zeroes can be estimated on the Riemann hypothesis by

\begin{eqnarray*}
\bigg| \sum_{\rho} \frac{2^{\rho+1}}{\rho (\rho+1)} \bigg| & < & 2^{3/2}  \sum_{\rho} \frac{1}{| \rho|^2} 
\end{eqnarray*}
where the value of this sum is explicitly known (see, for example, Davenport \cite{davenport}). The result follows.

\end{proof}

We now consider the existence of prime numbers in an interval of the form $(x-h, x+h)$. We do this by defining the weight function

\begin{displaymath}
   w(n) = \left\{
     \begin{array}{ll}
      1 - |n - x|/h & : \hspace{0.1in}  x-h < n < x+h\\
       0   & : \hspace{0.1in} \text{otherwise.}
     \end{array}
   \right.
\end{displaymath} 
and considering the identity

\begin{eqnarray} \label{weighted}
\sum_{n} \Lambda(n) w(n) & = & \frac{1}{h} \Big(\psi_1(x+h) - 2 \psi_1(x) + \psi_1 (x-h)\Big).
\end{eqnarray}
One can verify this by expanding the weight sum on the left hand side. An application of Lemma \ref{first} to the above equation gives us the following.

\begin{lem} \label{dog}
Let $x>0$ and $h>0$. Then

$$\sum_{n} \Lambda(n) w(n) = h - \frac{1}{h} \Sigma + \epsilon(h)$$
where 

$$\Sigma = \sum_{\rho} \frac{ (x+h)^{\rho+1} - 2x^{\rho+1} +(x-h)^{\rho+1}}{\rho(\rho+1)}$$
and 

$$|\epsilon(h)| < \frac{48}{5 h}.$$

\end{lem}

We use this lemma to prove our results. Our concern is for estimating the sum $\Sigma$, which we consider in two parts:

$$\Sigma = \Sigma_1 + \Sigma_2.$$
Here, $\Sigma_1$ ranges over the zeroes $\rho$ with $|\gamma| < \alpha x / h$, where $\alpha>0$ is to be chosen later, and $\Sigma_2$ is the contribution from the remaining zeroes.

\subsection{Proof of Theorem \ref{one}}

 For $\Sigma_1$, we notice that the summand may be written as

$$\int_{x-h}^{x+h} (h-|x-u|) u^{\rho-1} du,$$
the absolute value of which can be bounded above by

$$\frac{1}{\sqrt{x-h}} \int_{x-h}^{x+h} (h-|x-u|) du = \frac{h^2}{\sqrt{x-h}} .$$
It follows that

\begin{eqnarray*}
\Sigma_1 & \leq & \frac{h^2}{\sqrt{x-h}} \sum_{ |\gamma| < \alpha x / h}  1 \\
& = & \frac{2h^2}{\sqrt{x-h}} N(\alpha x/h)
\end{eqnarray*}
where $N(T)$ denotes the number of zeroes $\rho$ with $0 < \beta < 1$ and $0 < \gamma < T$. By Corollary 1 of Trudgian \cite{trudgianargument}, we have the bound

\begin{equation} \label{boundcount}
N(T) < \frac{T \log T}{2 \pi}
\end{equation}
for all $T>15$, and so

\begin{equation} \label{sigma1}
| \Sigma_1 | < \frac{\alpha x h}{\pi \sqrt{x-h}} \log(\alpha x/h)
\end{equation}
when $\alpha x/h > 15$. We can estimate $\Sigma_2$ trivially on the Riemann hypothesis by

\begin{eqnarray*}
|\Sigma_2| & < & 4 (x+h)^{3/2} \sum_{|\gamma| > \alpha x/h} \frac{1}{\gamma^2} \\
& = & 8 (x+h)^{3/2} \sum_{\gamma > \alpha x/h} \frac{1}{\gamma^2} \\
& < & \frac{4 h (x+h)^{3/2}}{\pi \alpha x} \log(\alpha x/h),
\end{eqnarray*}
where the last line follows from Lemma 1 (ii) of Skewes \cite{skewes}. Putting our estimates for $\Sigma_1$ and $\Sigma_2$ into Lemma \ref{dog} we have

\begin{eqnarray*}
\sum_{n} \Lambda(n) w(n) & > & h - \frac{1}{h} ( |\Sigma_1| + |\Sigma_2|) - \frac{48}{5h} \\
& = & h - \bigg(\frac{\alpha x}{\pi \sqrt{x-h}} +\frac{4 (x+h)^{3/2}}{\pi \alpha x} \bigg) \log(\alpha x/h) - \frac{48}{5h}.
\end{eqnarray*}
Notice that as we will choose $h$ to be $o(x)$, it follows that the term in front of the $\log$ is asymptotic to

$$ \Big(\frac{\alpha}{\pi} + \frac{4}{\pi \alpha} \Big) \sqrt{x}$$
It is a straightforward exercise in differential calculus to show that $\alpha = 2$ will minimise this term, and thus

\begin{eqnarray*}
\sum_{n} \Lambda(n) w(n) & > &  h - \frac{ 2}{\pi} \bigg(\frac{ x}{ \sqrt{x-h}} +\frac{ (x+h)^{3/2}}{ x} \bigg) \log(2 x/h) - \frac{48}{5h},
\end{eqnarray*}
or rather

\begin{eqnarray*} 
\psi(x+h) - \psi(x-h) & = & \sum_{x-h < n \leq x+h} \Lambda(n) \\
& > &  h - \frac{ 2}{\pi} \bigg(\frac{ x}{ \sqrt{x-h}} +\frac{ (x+h)^{3/2}}{ x} \bigg) \log(2 x/h) - \frac{48}{5h}.
\end{eqnarray*}

The sum on the left hand side of the above inequality is over prime powers. As such we consider the Chebyshev $\theta$-function given by

$$\theta(x) = \sum_{p \leq x} \log p.$$
Here we can use Theorem 14 and equation (5.5) of Schoenfeld \cite{schoenfeld} to get that

$$0.98 \sqrt{x} < \psi(x) - \theta(x) < 1.11 \sqrt{x} + 3 x^{1/3}$$
for all $x \geq 121$. We use this bound with our inequality for $\psi(x+h) - \psi(x-h)$ to get

\begin{eqnarray*}
\sum_{x - h < p \leq x + h} \log p & > &  h - \frac{ 2}{\pi} \bigg(\frac{ x}{ \sqrt{x-h}} +\frac{ (x+h)^{3/2}}{ x} \bigg) \log(2 x/h) \\
& & -1.11 \sqrt{x+h} - 3 (x+h)^{1/3} +0.98\sqrt{x-h} - \frac{48}{5h} .
\end{eqnarray*}
for this range of values. If we set $h = d \sqrt{x} \log x$, the leading term on the right hand side can be shown to be asymptotic to

\begin{equation} \label{asym}
\Big( d - \frac{2}{\pi} \Big) \sqrt{x} \log x + \frac{4}{\pi} \sqrt{x} \log \log x.
\end{equation}
Thus, for $d \geq 2/\pi$ we have that there is a prime in the interval

$$(x-d \sqrt{x}\log x, x+ d \sqrt{x} \log x]$$
and so we choose $d = 2/ \pi$. Then, using a monotonicity argument we have this for all $x \geq 65000$. Replacing $x + d \sqrt{x} \log x$ with $x$, we have that there is a prime in the interval

$$(x - 2 d \sqrt{x} \log x, x]$$
for all 

$$x \geq 65000 + \frac{2}{\pi} \sqrt{65000} \log(65000) \approx 66798.7$$
where $2d = 4/\pi$. This completes the proof of Theorem 2, for one can use \textsc{Mathematica} to verify the theorem for the remaining values of $x$.

\subsection{Proof of Theorem \ref{implicit}}

In what follows we show that the constant $4/\pi$ can be reduced to essentially 1 by a more detailed analysis of the sum $\Sigma_1$. Bounding trivially, we have that

$$|\Sigma_1| \leq x^{3/2} \sum_{|\gamma| < \alpha x/h} \frac{|(1+h/x)^{3/2} e^{i \gamma \log(1+h/x)} + (1-h/x)^{3/2} e^{i \gamma \log(1-h/x)} - 2|}{\gamma^2}.$$
By noting the straightforward bound

$$\log(1 \pm h/x) = \pm\frac{h}{x} + O\bigg(\frac{h^2}{x^2} \bigg)$$
which holds for $h = o(x)$, we have that

\begin{eqnarray*}
e^{i \gamma \log (1 \pm h/x)} & = & e^{\pm i \gamma h/x} \Big( 1 + O \Big( \gamma \frac{h^2}{x^2} \Big) \Big) \\
& = & e^{\pm i \gamma h/x} + O(\alpha h/x).
\end{eqnarray*}
as $|\gamma| < \alpha x/h$. Using this estimate and 

$$(1 \pm h/x)^{3/2} = 1 + O(h/x)$$
one obtains 

\begin{equation*}
|\Sigma_1| \leq 8 x^{3/2} \sum_{0<\gamma < \alpha x/h} \frac{ \sin^2 (\frac{h \gamma}{2 x}) }{\gamma^2} + O(\alpha  h \sqrt{x}).
\end{equation*}
This sum can be estimated using Theorem A from Ingham \cite{inghambook} and equation (\ref{boundcount}) to get that

$$|\Sigma_1| \leq \frac{4 x^{3/2}}{\pi}  \int_{\gamma_1}^{\alpha x/h} \frac{ \log(u) \sin^2 (\frac{hu}{2x})}{u^2}du + O(\alpha  h \sqrt{x})$$
where $\gamma_1 = 14.1347\ldots$ denotes the least positive ordinate of a zero. Employing the substitution $u =2xt/h$ and simplifying gives us that

$$| \Sigma_1 | \leq \bigg( \frac{2 }{\pi} \int_0^{\alpha/2} \frac{\sin^2 t}{t^2} dt \bigg) h \sqrt{x} \log(x/h)  + O(\alpha  h \sqrt{x}).$$ 

Now, estimating $\Sigma_2$ as in the previous section, we have from Lemma \ref{dog} and the above estimate for $\Sigma_1$ that

$$\sum_{x-h < n < x+h} \Lambda(n) \geq h - \Big( \frac{4}{\pi \alpha} + \frac{2}{\pi} \int_0^{\alpha/2} \frac{\sin^2 t}{t^2} dt\Big) \sqrt{x} \log(x/h)+O(\alpha \sqrt{x}).$$
If we set $h = c \sqrt{x} \log x$, and choose

$$c >  \frac{2}{\pi \alpha} + \frac{1}{\pi} \int_0^{\alpha/2} \frac{\sin^2 t}{t^2} dt,$$
then it follows that

$$\sum_{x-h < n < x+h} \Lambda(n) \gg \sqrt{x} \log x.$$
We note that we have $c = 1/2 + \epsilon$ provided that we take $\alpha$ to be sufficiently large. One can also remove the contribution of prime powers to the sum to have that there is a prime in the interval

$$(x-(1/2+\epsilon) \sqrt{x} \log x, x+(1/2+\epsilon) \sqrt{x} \log x)$$
for all sufficiently large values of $x$. This effectively completes the proof of Theorem \ref{implicit}. 

\subsection{A constant for Cram\'{e}r's theorem}

As mentioned in the introduction, one can also show that Theorem \ref{cramer} can be taken with $c=3+\epsilon$ provided that $x$ is sufficiently large. For if we take

$$c = 1+  \frac{2}{\pi \alpha} + \frac{1}{\pi} \int_0^{\alpha/2} \frac{\sin^2 t}{t^2} dt = \frac{3}{2} + \epsilon$$
then we have, again removing the contribution from prime powers, that

$$\sum_{x-h < p < x+h} \log p \geq \sqrt{x} \log x + O(\sqrt{x} \log \log x).$$
It remains to estimate by

\begin{eqnarray*}
\pi(x+h) - \pi(x-h) & > & \frac{1}{\log(x+h)} \sum_{x-h < p \leq x+h} \log p \\
& > & \sqrt{x} + O\bigg( \frac{\sqrt{x} \log \log x}{\log x} \bigg)
\end{eqnarray*}
and the result follows.

\section*{Acknowledgements}

The author would like to thank the referee for their feedback. In particular, for pointing out that the constant $4/\pi$ in Theorem \ref{one} could be replaced by $(1+\epsilon)$ through a more considered analysis.

\clearpage

\bibliographystyle{plain}

\bibliography{biblio}

\begin{thebibliography}{10}

\bibitem{cramer}
H.~Cram\'{e}r.
\newblock Some theorems concerning prime numbers.
\newblock {\em Arkiv Mathematik}, 5:1--33, 1920.

\bibitem{davenport}
H.~Davenport.
\newblock {\em Multiplicative Number Theory}.
\newblock Springer, Berlin, 1980.

\bibitem{goldston}
D.~Goldston.
\newblock On a result of {L}ittlewood concerning prime numbers.
\newblock {\em Acta Arithmetica}, 43(1):49--51, 1983.

\bibitem{goldstonheathbrown}
D.~R. Heath-Brown and D.~A. Goldston.
\newblock A note on the differences between consecutive primes.
\newblock {\em Mathematische Annalen}, 266(3):317--320, 1984.

\bibitem{inghambook}
A.~E. Ingham.
\newblock {\em The distribution of prime numbers}.
\newblock Number~30. Cambridge University Press, 1932.

\bibitem{ramaresaouter}
O.~Ramar{\'e} and Y.~Saouter.
\newblock Short effective intervals containing primes.
\newblock {\em Journal of Number Theory}, 98(1):10--33, 2003.

\bibitem{schoenfeld}
J.~B. Rosser and L.~Schoenfeld.
\newblock Approximate formulas for some functions of prime numbers.
\newblock {\em Illinois Journal of Mathematics}, 6(1):64--94, 03 1962.

\bibitem{schoenfeldjust}
L.~Schoenfeld.
\newblock Sharper bounds for the {C}hebyshev functions $\theta$(x) and
  $\psi$(x). {II}.
\newblock {\em Mathematics of Computation}, pages 337--360, 1976.

\bibitem{skewes}
S.~Skewes.
\newblock On the difference $\pi(x)- \text{li}(x)$ ({II}).
\newblock {\em Proceedings of the London Mathematical Society}, 3(1):48--70,
  1955.

\bibitem{titchmarsh}
E.C. Titchmarsh.
\newblock {\em The {T}heory of the {R}iemann {Z}eta-function}.
\newblock Oxford University Press, second edition, 1986.

\bibitem{trudgianargument}
T.~S. {Trudgian}.
\newblock {An improved upper bound for the argument of the {R}iemann
  zeta-function on the critical line}.
\newblock {\em Mathematics of Computation}, 81(278):1053--1061, 2012.

\bibitem{vonkoch}
H.~von Koch.
\newblock Sur la distribution des nombres premiers.
\newblock {\em Acta Mathematica}, 24:159, 1901.

\end{thebibliography}

\end{document}